\documentclass[a4paper]{article}
\pdfoutput=1

\usepackage[utf8]{inputenc}
\usepackage[babel=true,kerning,spacing]{microtype}
\usepackage{amssymb,amsfonts,amsmath,amsthm}
\usepackage{graphicx}
\usepackage{hyperref}
\usepackage{indentfirst}

\usepackage[english]{babel}

\newtheorem{theorem}{Theorem}[section]
\newtheorem{proposition}[theorem]{Proposition}
\newtheorem{corollary}[theorem]{Corollary}
\newtheorem{lemma}[theorem]{Lemma}
\newtheorem{algorithm}[theorem]{Algorithm}
\newtheorem*{definition*}{Definition}
\newcommand{\QQ}{\mathbb Q}
\newcommand{\ZZ}{\mathbb Z}
\newcommand{\NN}{\mathbb N}

\newcommand{\BB}{\mathcal B}
\newcommand{\norm}[1]{\operatorname{norm}{(#1)}}

\newcommand{\cl}[1]{\operatorname{cl}(#1)}

\newcommand{\ord}{\mathcal{O}}
\newcommand{\FF}[1]{\smash{\mathbb F_{#1}}}

\newcommand{\E}{\mathcal E}
\newcommand{\End}{\operatorname{End}}

\newcommand{\pp}{\mathfrak p}

\newcommand{\ff}[1]{\mathfrak f(#1)}
\newcommand{\aaa}{\mathfrak a}

\newcommand{\ZFV}{\ZZ[\pi,\overline\pi]}
\newcommand{\IOZFV}{[\ord_K:\ZFV]}

\newcommand{\Mat}[2]{\operatorname{Mat}_{#1}(#2)}
\newcommand{\Res}[2]{\operatorname{Res}_{#1}(#2)}

\title{Computing endomorphism rings\\of elliptic curves under the GRH}

\author{Gaetan Bisson}
\date{\small
   LORIA, 54506 Vand\oe{}uvre-l\`{e}s-Nancy, France
\\ TU/e, 5600 MB Eindhoven, The Netherlands}
\begin{document}

\maketitle

\begin{abstract}

We design a probabilistic algorithm for computing endomorphism rings of
ordinary elliptic curves defined over finite fields that we prove has a
subexponential runtime in the size of the base field, assuming solely the
generalized Riemann hypothesis.

Additionally, we improve the asymptotic complexity of previously known,
heuristic, subexponential methods by describing a faster isogeny-computing
routine.

\end{abstract}

%%%%%%%%%%%%%%%%%%%%%%%%%%%%%%%%%%%%%%%%%%%%%%%%%%%%%%%%%%%%%%%%
\section{Introduction}

Endomorphism rings of ordinary elliptic curves over finite fields are central
objects in complex multiplication (CM) theory; as such, they appear in various
computational number-theoretic contexts. For instance, the CM method for
generating curves with a prescribed number of points relies on evaluating
so-called Hilbert class polynomials, for which the state-of-the-art algorithm
of \cite{sutherland-hilbert} requires an endomorphism-ring-computing
subroutine. They are also potentially relevant security parameters in certain
cryptographic applications.

They were first studied by Kohel \cite{kohel} who, assuming the generalized
Riemann hypothesis (GRH), gave a deterministic method for computing them in
time $O(q^{1/3+\epsilon})$ where $q$ is the cardinality of the base field.
Recently, a probabilistic algorithm with subexponential complexity in $\log q$
was obtained in \cite{bisson-sutherland} by relying on several additional
assumptions; its runtime is
\[
L(q)^{\sqrt{3}/2+o(1)}
\qquad\text{where}\qquad
L(x)=\exp\sqrt{\log x\log\log x}.
\]

Here, we describe a variant of this method that computes endomorphism rings in
proven probabilistic subexponential time, assuming only the GRH; it ``ascends''
the lattice of orders in a generic manner, and ``tests'' orders using their
class group structure. The lattice-ascending procedure is suited to work in
general number fields, which is a necessary step for generalizing this
algorithm to higher-dimensional abelian varieties; for now, only the method of
Eisentr\"{a}ger and Lauter \cite{eisentrager-lauter} and that of Wagner
\cite{wagner} apply to this setting but they are both of exponential nature. To
prove the complexity of the order-testing method, we adapt material from Seysen
\cite{seysen} and proofs due to Hafner and McCurley \cite{hafner-mccurley} to
make use of a sharp bound derived from the GRH by Jao, Miller, and Venkatesan
\cite[Corollary~1.3]{expander-grh}.

Additionally, we use a more direct, faster isogeny-computing routine than
\cite{bisson-sutherland} which allows us to bring down the exponent in the
complexity. Explicitly, on input an ordinary elliptic curve $\E$ defined over a
finite field $\FF{q}$ our main algorithm outputs the structure of its
endomorphism ring $\End\E$ in proven (under the GRH) probabilistic time
\[
L(q)^{1+o(1)}+L(q)^{1/\sqrt{2}+o(1)}
\]
where the first term only accounts for the cost of factoring of a certain
integer less than $4q$ using the state-of-the-art proven method of Lenstra and
Pomerance \cite{lenstra-pomerance}; in other words, apart from that
factorization, we were able to adapt and prove under the GRH all parts of the
heuristic subexponential method above while improving its asymptotic
complexity.

Section~\ref{sec:context} fixes notations on endomorphism rings and orders.
Section~\ref{sec:relation} then presents the order-testing method using ``relations''.
Section~\ref{sec:iso} gives the direct-but-fast isogeny-computing routine.
Section~\ref{sec:ascend} describes our lattice-ascending procedure and main algorithm.
Section~\ref{sec:class} proves that class groups are characterized by short relations.
Section~\ref{sec:small} finally shows how orders are determined by their class groups.

%%%%%%%%%%%%%%%%%%%%%%%%%%%%%%%%%%%%%%%%%%%%%%%%%%%%%%%%%%%%%%%%
\section{Background}\label{sec:context}

Let $\E$ be an ordinary elliptic curve defined over a finite field $\FF{q}$.
The Frobenius endomorphism $\pi$ acts on geometric points of $\E$ by raising
their coordinates to the $q$\textsuperscript{th} power; its characteristic
polynomial $\chi_\pi(x)$ is of the form $x^2-tx+q$ and computing the integer
$t$ is equivalent to finding the number of points on the curve, namely
$\chi_\pi(1)$. Schoof showed in \cite{schoof} how this can be done in
deterministic polynomial time in the size of the base field, $\log q$.

Many endomorphisms stem from the Frobenius endomorphism, as Deuring proved in
\cite{deuring} that $\QQ\otimes\End\E\simeq\QQ(\pi)$. Since the number field
$K=\QQ[x]/(\chi_\pi(x))$ is isomorphic to $\QQ(\pi)$, by computing the trace
$t$ we have already determined the endomorphism ring ``up to fractions''. From
now on, we make this isomorphism implicit by setting $\pi=x$.\footnote{The
conjugate of $x$ might equivalently be taken as $\pi$; this choice just needs
to be made once and for all.}

The number field $K$ is called the CM field of $\E$; the implicit isomorphism
maps $\End\E$ to an order in $K$ so we have
\[
\ZZ[\pi]\subseteq\End\E\subseteq\ord_K
\]
where $\ord_K$ is the ring of integers of $K$. Conversely, Waterhouse proved in
\cite[Theorem~4.2]{waterhouse} that all orders containing $\ZZ[\pi]$ arise as
endomorphism rings. The index $[\ord_K:\ZZ[\pi]]$ is essentially the square
part of the discriminant $\Delta=t^2-4q$; this measures how broad the
search-range is: in the worst case, it can be exponential (in $\log q$).

The orders of $K$ containing $\ZZ[\pi]$ form a finite lattice (in the
set-theoretic sense) where $\ord_K$ is the maximal order, $\ZZ[\pi]$ the
minimal one, and $\End\E$ lies in between. Unfortunately it might have
exponentially many orders so we need to devise a better way of finding $\End\E$
than testing each in turn; this is the purpose of the lattice-ascending
algorithm of Section~\ref{sec:ascend} which tests only polynomially many
orders. For those orders $\ord$, we ``test'' whether $\ord\subseteq\End\E$ with
the methodology of Section~\ref{sec:relation} which we develop in
Sections~\ref{sec:class} and \ref{sec:small}.

%%%%%%%%%%%%%%%%%%%%%%%%%%%%%%%%%%%%%%%%%%%%%%%%%%%%%%%%%%%%%%%%
\section{The CM approach}\label{sec:relation}

We now present the approach of \cite{bisson-sutherland} to testing whether
$\ord\subseteq\End\E$ in a somewhat more abstract flavor. For the theory of
imaginary quadratic orders, we refer to \cite{cox}.

\bigskip

In this paper, it is implicitly understood that we exclusively consider ideals
of norm coprime to $\Delta$, so that their images in $\ZZ[\pi]$ are unramified
and invertible. Since every (invertible) ideal class of each order containing
$\ZZ[\pi]$ has a representative of this type, this has no effect on our use of
class groups, which arises from the following result of CM theory.

\begin{theorem}\label{th:cm}
When $\aaa$ is an ideal of $\End\E$, denote by $\phi_\aaa$ the isogeny with
kernel $\bigcap_{\alpha\in\aaa}\ker\alpha$. The ideal class group $\cl\ord$
acts faithfully and transitively on the set of isomorphism classes of elliptic
curves with endomorphism ring $\ord$ by $\aaa:\E\mapsto\phi_\aaa(\E)$.
\end{theorem}

Intuitively, the structure of the class group dictates that of the isogeny
graph; hence, by looking at the latter, we might deduce things on the former
and obtain information about the endomorphism ring. This action is effective,
as embodied in Proposition~\ref{prop:iso}. In this setting, we formalize the
notion of ``structure'' by the following concept.

\begin{definition*}
We define \emph{relations} as multisets of ideals of $\ZZ[\pi]$. We say that a
relation $R$ \emph{holds in an order $\ord$} (or that it \emph{is a relation of
$\ord$}) if the product $\prod_{\aaa\in R}\aaa\ord$ is trivial in $\cl\ord$;
we say that it \emph{holds in the isogeny graph} if the composition of the
isogenies $\phi_{\aaa\End\E}$ for $\aaa\in R$ fixes $\E$.
\end{definition*}

The theorem implies that a relation holds in $\End\E$ if and only if it holds
in the isogeny graph, which gives a way to tell the endomorphism ring apart
from other orders of the lattice (we will see in the next section that
$\phi_{\aaa\End\E}$ can be computed without knowing $\End\E$).

To avoid testing all orders, we rely on this simple result from
\cite[Chapter~7]{cox}:

\begin{lemma}\label{lemma:uprel}
If a relation holds in some order, it also does in all orders containing it.
\end{lemma}

Intuitively, as we ascend the lattice of orders, more and more relations hold,
which also translates into class groups getting smaller. This is why we chose
$\ZZ[\pi]$ to be the ring of our ideals: via the morphism $\aaa\mapsto\aaa\ord$
we can map ideals of $\ZZ[\pi]$ to any order above in a way that induces
surjective morphisms of class groups.

To search for the endomorphism ring $\End\E$ in the lattice, we will ``test''
whether orders $\ord$ lie below it by selecting relations of them and checking
whether they hold in the isogeny graph. Before we describe that procedure in
detail, let us mention how to compute isogenies.

%%%%%%%%%%%%%%%%%%%%%%%%%%%%%%%%%%%%%%%%%%%%%%%%%%%%%%%%%%%%%%%%
\section{Computing the CM action}\label{sec:iso}

To make use of Theorem~\ref{th:cm}, we need to work with \emph{isomorphism
classes} of elliptic curves; for this, we rely on \cite[Proposition~14.19]{cox}
which states that two ordinary elliptic curves are isomorphic if and only if
their cardinalities and $j$-invariants are the same. Computing the cardinality
takes polynomial time, and since the $j$-invariant is a rational function in
the coefficients of a Weierstrass equation, it does not take longer to evaluate
it. In the following, it is implicitly understood that we work with isomorphism
classes via this representation.

To evaluate the action $\phi_\aaa(\E)$ of an ideal $\aaa$, we combine classical
tools:

\begin{algorithm}\label{alg:cm-act}
~\nopagebreak

\noindent
\begin{tabular}{rl}
   \textsc{Input:} & An elliptic curve $\E/\FF{q}$
                    with Frobenius polynomial $\chi_\pi$ and an ideal $\aaa$.
\\ \textsc{Output:} & The isogenous elliptic curve $\phi_\aaa(\E)$.
\smallskip
\\ 1. & Find a basis $(P_i)$ of the $\ell$-torsion of $\E$ over $\FF{q^{\ell-1}}$ where $\ell=\norm\aaa$.
\\ 2. & Write the matrix $M$ of the Frobenius endomorphism on the basis $(P_i)$.
\\ 3. & Compute the eigenspaces of $M\in\Mat{2}{\ZZ/\ell\ZZ}$.
\\ 4. & Determine which is the kernel of the isogeny $\phi_\aaa$.
\\ 5. & Compute this isogeny.
\end{tabular}\hfill
\end{algorithm}

Step~5 computes $\phi_\aaa$ from its kernel, which V\'{e}lu's formul\ae{} \cite{velu}
do in $O(\ell)$ curve operations over $\FF{q^{\ell-1}}$. Step~4 relies on an
idea from the SEA algorithm found in \cite[Stage~3]{galbraith-hess-smart}:

\begin{proposition}
Let $\aaa$ be an ideal of $\ord$ of prime norm $\ell$; write it as
$\ell\ord+u(\pi)\ord$ where the polynomial $u$ is an irreducible factor of
$\chi_\pi\bmod\ell$. The characteristic polynomial of the restriction to the
kernel of $\phi_\aaa$ of the Frobenius endomorphism is $u$.
\end{proposition}

Since the map $\aaa\mapsto\aaa\ord$ from ideals of $\ZZ[\pi]$ preserves their
norm $\ell$ and polynomial $u$, there is no need to know $\ord$ to compute
$\phi_{\aaa\ord}$; this is particularly useful for $\ord=\End\E$.

Step~2 decomposes $\pi(P_i)$ as $\sum_{j\in{\{1,2\}}}M_{ij}P_j$ for which a
baby-step giant-step approach requires $O(\ell)$ operations in
$\E/\FF{q^{\ell-1}}$. Step~3 is classical and takes quasi-linear time in
$\log\ell$; it outputs the $\FF{q}$-rational subgroups of $\E[\ell]$ isomorphic
to $\ZZ/\ell\ZZ$.

Finally, Step~1 uses the fact that points of rational subgroups of order $\ell$
are necessarily defined over an extension of degree $\ell-1$; it proceeds by
selecting random $\ell^k$-torsion points over this extension and lifting one
along the other to obtain independent $\ell$-torsion points. This idea
originates from \cite[Theorem~1]{couveignes} to which we refer for details.

\begin{algorithm}
~\nopagebreak

\noindent
\begin{tabular}{rl}
   \textsc{Input:} & An elliptic curve $\E/\FF{q}$
                     with Frobenius polynomial $\chi_\pi$ and a prime $\ell$.
\\ \textsc{Output:} & A basis of the $\ell$-torsion $\E[\ell]$ of $\E$ over $\FF{q^{\ell-1}}$.
\smallskip
\\ a. & Decompose $\#\E(\FF{q^{\ell-1}})$ as $m\ell^k$ where $\ell\nmid m$.
\\ b. & Let $P$ and $Q$ be $m$ times random points of $\E(\FF{q^{\ell-1}})$;
\\ c. & Compute the order $\ell^{k_P}$ of $P$ and $\ell^{k_Q}$ of $Q$ and assume $k_P\geq k_Q$.
\\ d. & Precompute the table $(i,i\ell^{k_P-1}P)$ for $i\in\ZZ/\ell\ZZ$.
\\ e. & For $j$ from $k_Q-1$ down to $1$:
\\ f. & \qquad If $\ell^j Q=i\ell^{k_P-1}P$ for some $i$, set $Q\leftarrow Q-i\ell^{k_P-j-1}P$.
\\ g. & \qquad If $Q=0_\E$ then go back to Step~b.
\\ h. & Return $(\ell^{k_P-1}P,\ell^{k_Q-1}Q)$.
\end{tabular}\hfill
\end{algorithm}

The cardinality of $\E(\FF{q^{\ell-1}})$ can be computed as
$\Res{x}{\chi_\pi(x),x^{\ell-1}-y}(1)$; since it is $O(q^\ell)$, extracting
random points of it and multiplying them by $m$ requires $O(\ell\log q)$
operations in $\FF{q^{\ell-1}}$. Similarly, both $k_P$ and $k_Q$ are bounded by
$k=O(\ell\log q)$. The lookup in Step~f is negligible if an efficient data
structure such as a red-black tree is used to store the precomputed table of
Step~d. Finally, the probability of going back to Step~b is $O(1/\ell)$ as
proven in \cite{couveignes}.

Using fast arithmetic, operations in $\FF{q^{\ell-1}}$ take at most
$(\ell\log q)^{1+o(1)}$ time, so we have:

\begin{proposition}\label{prop:iso}
Algorithm~\ref{alg:cm-act} returns the curve $\phi_{\aaa\End\E}(\E)$ isogenous
to a prescribed curve $\E/\FF{q}$ in probabilistic time
$O(\ell^{2+o(1)}\log^{2+o(1)} q)$, where $\ell=\norm\aaa$.
\end{proposition}

%%%%%%%%%%%%%%%%%%%%%%%%%%%%%%%%%%%%%%%%%%%%%%%%%%%%%%%%%%%%%%%%
\section{Ascending the lattice of orders}\label{sec:ascend}

Orders in an imaginary quadratic field $K$ are of the form $\ZZ+f\ord_K$ for
some $f\in\NN$ known as the conductor; inclusion of orders corresponds to
divisibility of conductors. Those orders we are interested in contain
$\ZZ[\pi]$ so their conductors divide the index $[\ord_K:\ZZ[\pi]]$.

\bigskip

We will be ascending the lattice of orders one step at a time: each step
consists in enumerating all orders lying directly above a prescribed order,
that is, containing it with prime index $\ell$. The possible values for $\ell$
are the prime factors of $[\ord_K:\ZZ[\pi]]$ which can be listed by factoring
(the square-part of) the discriminant $\Delta$, for which the state-of-the-art
proven method of Lenstra and Pomerance \cite{lenstra-pomerance} uses
$L(q)^{1+o(1)}$ operations. Enumerating orders above (resp. below) then simply
amounts to dividing (resp. multiplying) the conductor by the possible $\ell$'s;
naturally, since our orders are to contain $\ZZ[\pi]$, this is subject to the
condition that the conductor remains a factor of the index $[\ord_K:\ZZ[\pi]]$.

Our strategy to locate the endomorphism ring in this lattice by testing orders
and ascending in corresponding directions works as follows: given some order
$\ord'$ contained in $\End\E$ (we start with $\ord'=\ZZ[\pi]$), find some order
$\ord$ directly above $\ord'$ which lies below $\End\E$; then replace $\ord'$
by $\ord$ and iterate the process. The ascension ends when no $\ord$ is
contained in $\End\E$; then, we must have $\End\E\simeq\ord'$. See
Figure~\ref{fig:ord-lattice} where we start from the bottom and ascend towards
orders $\ord$ for which the statement $\ord\subseteq\End\E$ holds.

\begin{figure}
\begin{center}
\includegraphics[width=6.6cm]{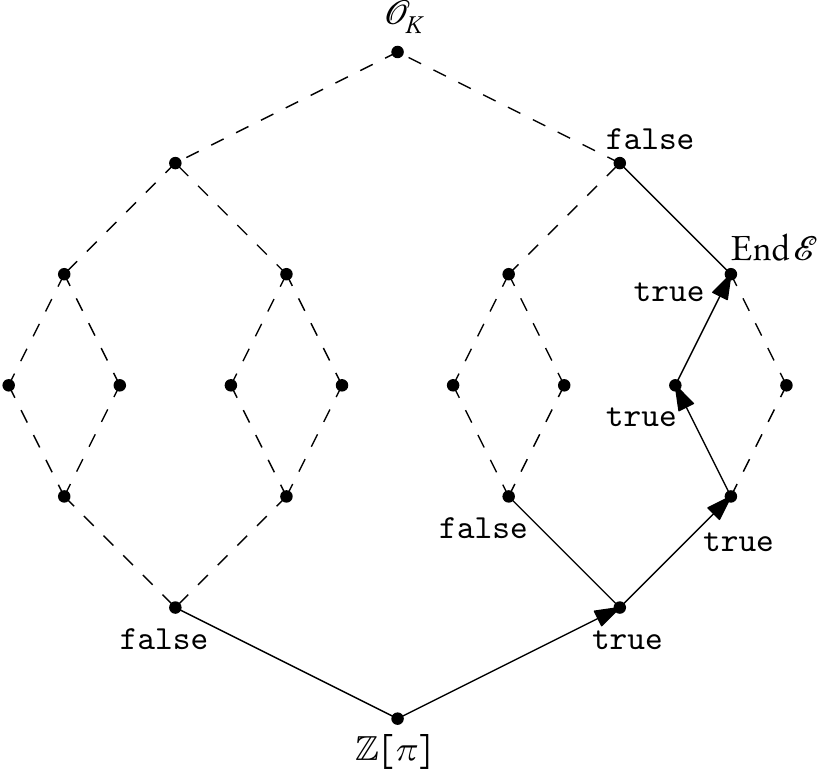}
\end{center}
\caption{Locating $\End\E$ by ascending a test-sequence of orders.}
\label{fig:ord-lattice}
\end{figure}

We formalize this procedure into:

\begin{algorithm}
~\nopagebreak

\noindent
\begin{tabular}{rl}
   \textsc{Input:} & An ordinary elliptic curve $\E$ over a finite field $\FF{q}$.
\\ \textsc{Output:} & An order isomorphic to the endomorphism ring of $\E$.
\smallskip
\\ 1. & Compute the Frobenius polynomial $\chi_\pi(x)$ of $\E$.
\\ 2. & Factor the discriminant $\Delta$ and construct the order $\ord'=\ZZ[\pi]$.
\\ 3. & For orders $\ord$ directly above $\ord'$:
\\ 4. & \qquad If $\ord\subseteq\End\E$ set $\ord'\leftarrow\ord$ and go to Step~3.
\\ 5. & Return $\ord'$.
\end{tabular}\hfill
\end{algorithm}

Steps~1 and 2 are classical and only require polynomial time in $\log q$,
except the factorization of $\Delta$ which takes $L(q)^{1+o(1)}$ time.
Under the GRH, we will later prove:

\begin{proposition}[GRH]\label{prop:test}
Let $\ord$ be an order above $\ZZ[\pi]$. One can determine whether
$\ord\subseteq\End\E$ in probabilistic time $L(q)^{1/\sqrt{2}+o(1)}$
with failure probability $o(1/log^2 q)$.
\end{proposition}

The number of orders directly above $\ZZ[\pi]$ (to be tested in Step~4) is the
number of prime factors of $[\ord_K:\ZZ[\pi]]$ and it decreases as $\ord'$
grows; the number of ascending steps (of times Step~3 is reached) is bounded by
the sum of the exponents in the factorization of $[\ord_K:\ZZ[\pi]]$ into prime
powers. These two quantities are smaller than $\log_2\Delta$ so the overall
number of tests is at most quadratic in $\log q$. As a consequence, we have:

\begin{theorem}[GRH]
The endomorphism ring of an ordinary elliptic curve defined over $\FF{q}$ can
be computed, with failure probability $o(1)$, in probabilistic time
$L(q)^{1+o(1)}+L(q)^{1/\sqrt{2}+o(1)}$ where the first term only accounts for
the complexity of factoring the discriminant $\Delta=O(q)$.
\end{theorem}

The output may be unconditionally verified using the certification method of
\cite[Section~3.2]{bisson-sutherland}. This probabilistic procedure can be
adapted to use the isogeny-computing routing of Section~\ref{sec:iso} and the
proof material of Section~\ref{sec:class}; under the GRH, it then requires
$L(q)^{1/\sqrt{2}+o(1)}$ operations. As a result, we obtain an algorithm for
which the above theorem holds without the ``failure probability'' statement;
this is sometimes called a \emph{Las Vegas} algorithm.

\bigskip

The rest of this paper is devoted to the proof of Proposition~\ref{prop:test}.

%%%%%%%%%%%%%%%%%%%%%%%%%%%%%%%%%%%%%%%%%%%%%%%%%%%%%%%%%%%%%%%%
\section{Class groups from short relations}\label{sec:class}

To test whether $\ord\subseteq\End\E$ reliably, we \emph{characterize} $\ord$
by a set of relations $R$ that hold in it but not collectively in any order of
the lattice not containing it. We will then test whether they hold in the
isogeny graph, so we seek relations $R$ for which the (quasi-\nobreak)quadratic
cost of computing the associated isogeny $\sum_{\aaa\in R}\norm\aaa^2$ is
small.

We start by bounding the norms of ideals to appear in our relations: form the
set $\BB$ of prime ideals $\pp$ of $\ZZ[\pi]$ with norm less than some integer
$N$ to be fixed later, and consider \emph{smooth} ideals
\[
\sigma(n)=\prod_{\pp\in\BB} \pp^{n_\pp}
\]
for vectors $n\in\ZZ^\BB$. If $\sigma_\ord(n)$ denotes the corresponding ideal
class in $\cl\ord$, the kernel of the map $\sigma_\ord$ is a lattice
$\Lambda_\ord$ in $\ZZ^\BB$ consisting of all relations of $\ord$ formed of
ideals in $\BB$: the coordinate $n_\pp$ is the multiplicity of the ideal $\pp$
in the relation. When $\sigma_\ord$ is surjective, we have
\[
\cl\ord\simeq\ZZ^\BB/\Lambda_\ord.
\]
Nothing of value is lost by only considering relations $R$ of $\Lambda_\ord$
since, assuming the GRH, Bach proved in \cite{bach-erh} that $\sigma_\ord$ is
indeed surjective provided that $N\geq 12\log^2|\Delta|$.

The isogeny chain associated to a relation $n\in\Lambda_\ord$ comprises at most
$\|n\|_1=\sum|n_\pp|$ isogenies of degree up to $N$ so the complexity of
evaluating it is crudely bounded by $\|n\|_1 N^{2+o(1)}$. This norm can be
controlled by a result of Jao, Miller, and Venkatesan
\cite[Corollary~1.3]{expander-grh} and more specifically its following
specialization found in \cite[Theorem~2.1]{quantum-iso}.

\begin{theorem}\label{th:grh}
Under the GRH, for all positive numbers $\epsilon$ there exists a constant
$c>1$ such that, for any imaginary quadratic order $\ord$ of discriminant $D$
and integers $N\geq\log^{2+\epsilon}|D|$ and
\[
l\geq c\frac{\log|D|}{\log\log|D|},
\]
the probability, for random vectors $n\in\ZZ^\BB$ of norm $l$, that the ideal
class $\sigma_\ord(n)$ falls in any subset $S$ of $\cl\ord$ is at least
$\frac{1}{2}\frac{\#S}{\#\cl\ord}$.
\end{theorem}

\begin{corollary}[GRH]
For $N=\log^{2+\epsilon}|D|$ the diameter of the lattice $\Lambda_\ord$ is
$o(\log^{4+\epsilon}|D|)$.
\end{corollary}

\begin{proof}
To prove this, we construct a generating set for $\Lambda_\ord$ formed by
$O(\log^{2+\epsilon}|D|)$ relations of norm $o(\log^2|D|)$. Siegel showed in
\cite{siegel} that $\cl\ord$ is an abelian group of order $D^{1/2+o(1)}$
so there exist $O(\log|D|)$ ideal classes $\alpha_i$ such that
$\ZZ^\BB/\Lambda_\ord\simeq\prod\left<\alpha_i\right>$; we fix these and
proceed to write a generating set for $\Lambda_\ord$ consisting of:
\begin{itemize}
\item relations expressing that $\alpha_i^{\operatorname{ord}(\alpha_i)}=1$;
\item relations expressing the primes $\pp\in\BB$ in terms of the $\alpha_i$.
\end{itemize}
First define a map $\sigma_\ord^{-1}$ by fixing a preimage of norm at most
$c\log|D|/\log\log|D|$ for each ideal class; it exists by Theorem~\ref{th:grh}.
Now use a double-and-add approach to ensure that norms remain small: for each
$i$, express that $\alpha_i^{\smash{\operatorname{ord}(\alpha_i)}}=1$ by the
relations
\begin{enumerate}
\item[(i)]
$\sigma_\ord^{-1}\left(\alpha_i^{2^j}\right)-2\sigma_\ord^{-1}\left(\alpha_i^{2^{j-1}}\right)$
for $j\in\{1,\ldots,\lfloor\log_2\operatorname{ord}(\alpha_i)\rfloor\}$;
\item[(ii)]
$\sum_j b_j\sigma_\ord^{-1}\left(\alpha_i^{2^j}\right)$
where $b_j$ denotes the $j$\textsuperscript{th} least significant bit
of $\operatorname{ord}(\alpha_i)$.
\end{enumerate}
Now write each $\pp\in\BB$ on the $\alpha_i$ by decomposing its class as a
product $\prod\alpha_i^{n_i}$ where
$n_i\in\{0,\ldots,\operatorname{ord}(\alpha_i)\}$; noting $\delta_\pp$ the
vector with coordinate one at $\pp$ and zero elsewhere, this gives the relations:
\begin{enumerate}
\item[(iii)] $\delta_\pp-\sum_i\sum_j c_{ij}\sigma_\ord^{-1}\left(\alpha_i^{2^j}\right)$
where $c_{\smash{ij}}$ is the $j$\textsuperscript{th} least significant bit of $n_i$.
\end{enumerate}
Preimages by $\sigma_\ord$ have length $o(\log|D|)$ and there are at most
$\sum\lfloor\log_2\operatorname{ord}(\alpha_i)\rfloor=O(\log|D|)$ terms,
therefore each such relation has length $o(\log|D|)^2$.
\end{proof}

To generate short relations, we simply plug this bound into the algorithm of
Seysen \cite{seysen} and rely on ingredients of Hafner and McCurley
\cite{hafner-mccurley} for the proof. Note that Childs, Jao, and Soukharev
\cite{quantum-iso} proposed a similar algorithm for finding one relation, while
we seek several random relations in order to characterize the order $\ord$.

\begin{algorithm}
~\nopagebreak

\noindent
\begin{tabular}{rl}
   \textsc{Input:} & An imaginary quadratic order $\ord$ of discriminant $D$.
\\ \textsc{Output:} & A quasi-random relation $n\in\Lambda_\ord$ with $\|n\|_1=o(\log^{6+\epsilon}|D|)$.
\smallskip
\\ 1. & Form the set $\BB$ of primes $\pp$ of $\ord$ with norm less than $N=L(q)^z$.
\\ 2. & Draw uniformly at random a vector $x\in\ZZ^\BB$ with coordinates
\\    & $|x_\pp|<\log^{4+\epsilon}|D|$ if $\norm{\pp}<\log^{2+\epsilon}|D|$, else $x_\pp=0$.
\\ 3. & Compute the reduced ideal representative $\aaa$ of $\sigma_\ord(x)$.
\\ 4. & If $\aaa$ factors over $\BB$ as $\prod \pp^{y_\pp}$ then return the vector $x-y$.
\\ 5. & Otherwise, go back to Step~2.
\end{tabular}\hfill
\end{algorithm}

\begin{proposition}[GRH]\label{prop:findrel}
Let $\ord$ be an order containing $\ZZ[\pi]$; its discriminant $D$ is then at
most $\Delta=O(q)$. The algorithm above requires
$L(q)^{z+o(1)}+L(q)^{1/(4z)+o(1)}$ operations to find a relation of $\ord$ whose
associated isogeny can be computed in time $L(q)^{2z+o(1)}$.
\end{proposition}

\begin{proof}
Step~4 consists in testing the smoothness of (the norm of) $\aaa$;
Lenstra, Pila, and Pomerance \cite[Corollary~1.2]{lenstra-pila-pomerance}
proved this requires $\exp\left(\log^{2/3+o(1)}N\right)\log^3 q$ operations,
that is, $L(q)^{o(1)}$ since $N=L(q)^z$. The probability that this
factorization is successful, in other words, that the norm of $\aaa$ is
$N$-smooth is $L(q)^{1/(4z)+o(1)}$ provided that it behaves as a random
integer; this follows directly from combining the corollary above with
\cite[Proposition~4.4]{seysen}; see also \cite{hafner-mccurley}. The relation
involves $o(\log^{4+2+\epsilon} q)$ ideals of norm up to $L(q)^z$, whence the
time bound for evaluating the associated isogeny by Proposition~\ref{prop:iso}.
\end{proof}

Hopefully, the relations we generate discriminate between orders with distinct
class groups:

\begin{lemma}[GRH]\label{lem:sub}
Take any two orders $\ord$ and $\ord'$; a relation of $\ord$ generated by
the algorithm above has a probability
$[\Lambda_\ord:\Lambda_\ord\cap\Lambda_{\ord'}]^{-1}+o(1)$ of also holding in
$\ord'$.
\end{lemma}

\begin{proof}
This follows directly from \cite[Lemma~2]{hafner-mccurley} adapted to the
context of our algorithm, which proves the quasi-randomness of the relations it
generates.
\end{proof}

%%%%%%%%%%%%%%%%%%%%%%%%%%%%%%%%%%%%%%%%%%%%%%%%%%%%%%%%%%%%%%%%
\section{Orders from class groups}\label{sec:small}

Our proof of Proposition~\ref{prop:test} now boils down to exhibiting the
following.

\begin{algorithm}
~\nopagebreak

\noindent
\begin{tabular}{rl}
   \textsc{Input:} & An ordinary elliptic curve $\E/\FF{q}$ and an order $\ord\supseteq\ZZ[\pi]$.
\\ \textsc{Output:} & Whether $\ord\subseteq\End\E$.
\smallskip
\\ 1. & Compute a set of $3\log\log q$ relations of $\ord$.
\\ 2. & If one does not hold in the isogeny graph, return \texttt{false}.
\\ 3. & Check whether $\ord\subseteq\End\E$ locally at $2$ and $3$; if not, return \texttt{false}.
\\ 4. & Return \texttt{true}.
\end{tabular}\hfill
\end{algorithm}

By Proposition~\ref{prop:findrel}, Step~1 requires
$L(q)^{z+o(1)}+L(q)^{1/(4z)+o(1)}$ operations to find relations whose
associated isogenies are then evaluated by Step~2 in $L(q)^{2z+o(1)}$. To
balance these quantities, we set $z=1/2\sqrt{2}$ which gives an overall
complexity of $L(q)^{1/\sqrt{2}+o(1)}$.

The correctness follows from Lemma~\ref{lemma:uprel} and Theorem~\ref{th:cm},
in that Steps~1 and 2 determine whether
$\Lambda_\ord\subseteq\Lambda_{\End\E}$; the probability of failure is at most
$(2+o(1))^{-3\log\log q}=o(1/\log^2q)$, by Lemma~\ref{lem:sub} applied to
$\ord'=\End\E$. The proposition below argues that, combined with Step~3, this
really determines whether $\ord\subseteq\End\E$.

\begin{proposition}\label{prop:unique}
Let $\ord$ and $\ord'$ be two orders in an imaginary quadratic field $K$. The
lattice $\Lambda_{\ord'}$ contains $\Lambda_\ord$ if and only if the order
$\ord'$ contains $\ord$ or:
\begin{enumerate}
\setlength{\itemsep}{-2pt}
\item $K=\QQ(\sqrt{-4})$ and $\ord'$ has conductor $2$;
\item $K=\QQ(\sqrt{-3})$ and $\ord'$ has conductor $2$ or $3$;
\item The prime $2$ splits in $K$ and $\ord'$ has index $2$ in some order above $\ord$ of odd conductor.
\end{enumerate}
\end{proposition}

Intuitively, this means that identifying orders by their class groups has a
single blind spot locally at $2$ and $3$ where the two biggest orders cannot be
distinguished; Step~3 is thus required in our algorithm to ensure it exactly
determines the endomorphism ring even amongst those orders with identical class
groups. This statement is a straightforward refinement of
\cite[Proposition~5]{bisson-sutherland}; we nevertheless give the proof below
for completeness.

\begin{proof}
Denote by $S_\ord$ (resp. $S_{\ord'}$) the set of primes $\ell$ that split into
principal ideals in $\ord$ (resp. $\ord'$). Using relations formed of a single
prime ideal, we see that $\Lambda_\ord\subseteq\Lambda_{\ord'}$ implies
$S_\ord\subseteq S_{\ord'}$. Now $S_\ord$ (resp. $S_{\ord'}$) is also the set
of primes that split completely in the ring class field $L_\ord$ of $\ord$
(resp. $L_{\ord'}$). By Chebotarev's density theorem $S_\ord\subseteq
S_{\ord'}$ thus implies $L_{\ord'}\subseteq L_\ord$ which means that the class
field theory conductor $\ff{L_{\ord'}/K}$ of $L_{\ord'}$ divides
$\ff{L_\ord/K}$.

This conductor $\ff{L_\ord/K}$ is related to that $f_\ord$ of $\ord$
as follows (see \cite[Exercises~9.20--9.23]{cox}).
\[\ff{L_\ord/K}=\begin{cases}
\ord_K, & \text{when $K=\QQ(\sqrt{-4})$ and $f_\ord=2$}, \\
\ord_K, & \text{when $K=\QQ(\sqrt{-3})$ and $f_\ord=2$ or $3$}, \\
u\ord_K, & \text{when $2$ splits in $K$ and $f_\ord=2u$ with $u$ odd}, \\
f_\ord\ord_K, & \text{otherwise}.
\end{cases}
\]
Naturally, the same stands for $\ord'$. In the latter case, the fact that
$\ff{L_\ord/K}$ divides $\ff{L_{\ord'}/K}$ implies that $f_{\ord'}$ divides
$f_\ord$, in other words $\ord\subseteq\ord'$; the three other cases
correspond, in order, to the exceptions listed in the proposition.
\end{proof}

Finally, let us address Step~3. To check whether $\ord\subseteq\End\E$ locally
at some prime $p$, one uses a method of Kohel \cite{kohel} known as ``climbing
the volcano'', which can be done in the traditional ``blind'' way by following
three $p$-isogeny paths from $\E$ and seeing which hits the ``floor of
rationality'' first, or using the more advanced technique of \cite{ionica-joux}
to directly determine the kernel of the ascending $p$-isogeny by pairing
computations. Eventually, both methods return the valuation at $p$ of the
conductor of $\End\E$ by computing at most
$O(\operatorname{val}_p[\ord_K:\ZZ[\pi]])$ isogenies of degree $p$; since we
use $p=2,3$, this takes polynomial time in $\log q$.

\section*{Acknowledgments}

This work owes everything to the numerous fruitful discussions the author had
with Karim Belabas, Andreas Enge, Pierrick Gaudry, Andrew V. Sutherland, and
Emmanuel Thom\'{e}. This paper also benefited from the remarks of Tanja Lange
who kindly proofread it.

\bibliography{document}

\end{document}